\documentclass[twoside,11pt,a4 paper]{article}
\usepackage{amssymb,amsmath,latexsym,amsthm}
\usepackage{graphics}
\usepackage{amssymb}
\usepackage{makeidx}
\usepackage{color}
\usepackage{eucal}
\newtheorem{pro}{Proposition}[section]
\newtheorem{teo}[pro]{Theorem}
\newtheorem{obs}[pro]{Remarks}
\newtheorem{ob}[pro]{Remark}
 \newtheorem{cor}[pro]{Corollary}

\newtheorem{defin}[pro]{Definition}

\newtheorem{nton}[pro]{Notation}

\begin{title}
{\bf The C-property for slice regular functions and applications to the Bergman space}
\end{title}

\author{Fabrizio Colombo\\Politecnico di
Milano\\Dipartimento di Matematica\\Via E. Bonardi, 9\\20133 Milano,
Italy\\fabrizio.colombo@polimi.it
\and
J. Oscar Gonz\'alez-Cervantes\\
 Departamento de Matem\'aticas  \\
 E.S.F.M. del
I.P.N. 07338 \\
M\'exico D.F., M\'exico
\\
jogc200678@gmail.com
\and
Irene Sabadini\\Politecnico di
Milano\\Dipartimento di Matematica\\Via E. Bonardi, 9\\20133 Milano,
Italy\\irene.sabadini@polimi.it }

\date{  }
\begin{document}
\maketitle

\begin{abstract}
This paper has a twofold purpose: on one hand we deepen the study of slice regular functions by studying their behavior
with respect to the so-called C-property and anti-C-property. We  show that, for any fixed basis of the algebra of quaternions $\mathbb{H}$  any slice regular function decomposes into the sum of four slice regular components each of them satisfying the C-property.
Then, we will use these results to show a reproducing property  of the Bergman kernels of  the second kind.
\end{abstract}

\section{Introduction}
In this paper we continue the study of the Bergman theory in the setting of slice regular functions started in \cite{CGESS}.  After their definition in \cite{advances} which has been generalized to the case of functions with values in a Clifford algebra in \cite{slicecss}, these functions have been intensively studied during the past years and their theory, as well as the application to the functional calculus, is treated in the book \cite{cssbook}.
In \cite{CGESS} we showed that the Bergman theory admits two possible formulations in this framework: the so-called formulations of the first and of the second kind.
In both cases, the open sets $\Lambda$ on which we construct the theory should be suitable domains.  The two formulations differ since the first relies on an integral formula computed on $\Omega$ while in the second case the integral is computed on the intersection of $\Lambda$ with a complex plane $\mathbb{C}_{\bf i}=\{z=x+\mathbf{i}y,\ x,y\in\mathbb{R}\}$. The Bergman spaces in the two formulations are equipped with different inner products.
Thus, it is natural to ask whether the two theories can be related and in this paper we show that the Bergman kernel of the second kind, which is characterized by a reproducing property on the complex slices, can reproduce a function on an open set in $\mathbb{H}$ by means of a suitable operator that we denote by $M_{\bf i}$ which is studied in the last section of the article. The technical tools that we use are based on some properties of slice regular functions which depend on their behavior with respect to the quaternionic conjugation. These properties, in turn, are related to some analogous properties of holomorphic functions. More precisely, let $\Omega$ be an open set in the complex plane which is invariant under conjugation, i.e. such that $\bar z\in\Omega$ for all $z\in\Omega$. Then, we consider complex valued functions $f$ defined on $\Omega$ such that $f(\bar z)=\overline{f(z)}$ or such that $f(\bar z)=-\overline{f(z)}$. In the first case, we say that the function $f$ satisfies the C-property, ($f$ is also called intrinsic function, see \cite{gp}). In the second case we say that $f$ satisfy the anti-C-property  ("C" stands for conjugation).
These properties can be formulated also in the quaternionic setting by considering the quaternionic conjugation.

We will show that holomorphic functions satisfying the C-property or anti-C-property can be extended to slice regular functions with the analogous property. Moreover, we will show a refined version of the Splitting lemma which allow to decompose the restriction of slice regular function to a complex plane into a sum of four (intrinsic) holomorphic functions. Moreover, we prove that for any fixed basis of the algebra of quaternions $\mathbb{H}$, any slice regular function decomposes canonically into the sum of four slice regular functions satisfying the C-property (intrinsic, for short).
Finally, we study the subsets of the Bergman spaces with respect to the C-property or anti-C-property. We decompose the
Bergman kernel of the second kind in its four slice regular intrinsic components and we show their reproducing property. Finally, we discuss the relations among integral representation computed on  {an} open set in $\mathbb{H}$ and on the slices.
We finally mention the classical references for the Bergman theory in the complex case \cite{berg, bgsc}
and for the sake of completeness some papers that treat the Bergman theory in the hypercomplex setting
\cite{bradel, bds, const, conskrau, conskrau2, delanghe, SG2, shavas, shavas2, shavas3}.

\section{The holomorphic Bergman space and the C-property}
In this section we will work with the  {Bergman theory in one complex variable}. We start by recalling some notations and then we will discuss the so called C-property. Let $\Omega$ be an open set of the complex plane which is invariant under conjugation. A function $f$ defined on $\Omega$ is said to have the C-property if $f(\bar z)=\overline{f(z)}$ while it is said to have the anti-C-property if $f(\bar z)=-\overline{f(z)}$.  It can be shown that any complex function decomposes into the sum of two functions $f_1(z)$ and $f_2(z)$ having the C-property. An analogous decomposition holds with two functions  {satisfying}  the anti-C-property.
This fact has some consequences at the level of the Bergman space. We discuss some immediate consequences of this decomposition, by inserting the proofs of our results for the sake of completeness.

\begin{defin}
 Let us denote the complex conjugation  on $\mathbb C$ by $Z_{\mathbb C}(z)=\bar z,\quad \forall z\in \mathbb C$. We will say that a domain $\Omega\subset\mathbb C$ is Z-invariant if and only if  $Z_{\mathbb C}(\Omega)=\Omega$.
\end{defin}

\begin{defin}
Let $\Omega\subseteq\mathbb{C}$ be an open set.
 The space of holomorphic functions on $\Omega$ is denoted by $Hol(\Omega)$, while   by $\mathcal A^2(\Omega)$, and $\mathcal K_\Omega(\cdot,\cdot)$ we mean, respectively,  the  usual complex Bergman space and its Bergman kernel, both associated to $\Omega$.
\end{defin}

In this section, the domains $\Omega$ we consider are supposed to  { be  Z-invariant}. Note that  using the Cauchy-Riemann equations, one obtains  that   $f\in Hol(\Lambda)$  if and only if  $Z_{\mathbb C}\circ f\circ Z_{\mathbb C}\in Hol(\Lambda)$.

Let $\mathbb B^2$ be the unit disk. In particular,  by the Bergman theory,  we have that
$$f(z)=\int_{\mathbb B^2} \frac{1}{\pi(1- \bar z\zeta)^2 }f(\zeta)     {d\sigma_{\zeta}},$$
for any  $f\in \mathcal A^2(\mathbb B^2)$ and $z\in \mathbb B^2$. By setting $ \zeta= \bar v$,  $z=\bar w$ one obtains that
$$\overline{ f(\bar w)}=\int_{\mathbb B^2} \frac{1}{\pi(1-  \bar w  v )^2 }\overline{  f(\bar v)} d\sigma_v ,$$
that is $Z_{\mathbb C}\circ f\circ Z_{\mathbb C}\in \mathcal A^2(\mathbb B^2)$.

In general, as the complex conjugation $Z_\mathbb C$ preserves the norm  of the complex numbers and the differential element of area,  then $f\in \mathcal A^2(\Omega)$  if and only if
$Z_{\mathbb C}\circ f\circ Z_{\mathbb C}\in \mathcal A^2(\Omega)$.

Moreover, we have
$$
  \int_\Omega | f(\zeta)|^2 d\sigma_{\zeta} =  \int_\Omega | Z_{\mathbb C}\circ f\circ Z_{\mathbb C}(\zeta)   |^2 d\sigma_{\zeta} , \quad \forall f\in  \mathcal A^2(\Omega).$$

\begin{defin} Let $\Omega\subseteq\mathbb{C}$ be a Z-invariant open set. We define the following sets:
$$Hol_c(\Omega):=\{f\in Hol(\Omega) \ \mid \ Z_{\mathbb C}\circ f\circ Z_{\mathbb C}=f\},$$
$$Hol_{\bar c}(\Omega):=\{f\in Hol(\Omega) \ \mid \ -Z_{\mathbb C}\circ f\circ Z_{\mathbb C}=f\},$$
$$\mathcal A_{c}^2(\Omega):= \{f\in \mathcal A^2(\Omega) \ \mid   f= Z_{\mathbb C}\circ f\circ Z_{\mathbb C} \} ,$$
$$\mathcal A_{\bar c}^2(\Omega):= \{f\in \mathcal A^2(\Omega) \ \mid  f=-  Z_{\mathbb C}\circ f\circ Z_{\mathbb C} \} ,$$
and we will say that a holomorphic function  $f$ defined on   $\Omega$ satisfies the C-property  if and only if $f\in Hol_{c}(\Omega)$. We will say that a holomorphic function  $f$ defined on   $\Omega$ satisfies the anti-C-property  if and only if $f\in Hol_{\bar c}(\Omega)$.

\end{defin}

\begin{ob}\label{ob1} \ {} {\rm
\begin{enumerate}
  \item Note that $Hol_{c}(\Omega) $ and $Hol_{\bar c}(\Omega) $ are real-linear spaces  { of $\mathbb{C}$-valued functions}.  Since for any $f,g\in Hol_{c}(\Omega)$ the following property holds
$$
Z_{\mathbb C}\circ (f \lambda + g)\circ Z_{\mathbb C}=Z_{\mathbb C}\circ f \circ Z_{\mathbb C}\lambda + Z_{\mathbb C}\circ  g\circ Z_{\mathbb C}= f \lambda + g,
$$
for any $\lambda \in \mathbb R$. A similar reasoning can be used for $Hol_{\bar c}(\Omega)$.\\

Moreover,  the sets  $\mathcal A_{c}^2(\Omega)$ and $\mathcal A_{\bar c}^2(\Omega)$ equipped with the norm inherited by $\mathcal A^2(\Omega)$ are real-linear {spaces of $\mathbb C$-valued functions}.

  \item Any polynomial  with real coefficients belongs to $\mathcal A_{c}^2(\Omega)$. On the other hand,  any polynomial  whose coefficients are  pure imaginary complex numbers belongs to $\mathcal A_{\bar c}^2(\Omega)$.

  \item Finally, it is easy to see that the mapping $f\mapsto i f$ is an isomorphism  {between} the real-linear spaces $Hol_{c}(\Omega)$ and $Hol_{\bar c}(\Omega)$. The restriction of this mapping to the Bergman space is an isometric isomorphisms between     $\mathcal A_c^2(\Omega)$ and  $\mathcal A_{\bar c}^2(\Omega)$.

\end{enumerate}
}
\end{ob}

 {The proof  of the following result is given in \cite{global}. }
\begin{pro}
\label{ob3}
Let $\Omega\subseteq\mathbb{C}$ be a Z-invariant open set.
Given any  $f\in Hol(\Omega) $, there exists a unique pair of functions $f_1\in  Hol_{c}(\Omega)$ and $f_2\in Hol_{\bar c}(\Omega)$ such that $f=f_1 + f_2$.
\end{pro}

 As  immediate corollaries we obtain:
\begin{cor} Let $\Omega\subseteq\mathbb{C}$ be a Z-invariant open set. We have
$$
 Hol(\Omega) = Hol_c(\Omega) \oplus Hol_{\bar c}(\Omega)
$$
and, analogously,
$$
\mathcal A^2(\Omega)=\mathcal A_{c}^2(\Omega)\oplus\mathcal A_{\bar c}^2(\Omega).
$$
\end{cor}

\begin{cor}\label{pro0}
Let $\Omega\subseteq\mathbb{C}$ be a Z-invariant open set.
  Given any $f\in Hol(\Omega)$ there  {exists} a unique pair $f_1, f_2\in Hol_{c}(\Omega)$ such that
  $$f=f_1+i f_2,$$
  i.e.
\begin{equation}\label{hol-direct-sum}
 Hol(\Omega) = Hol_c(\Omega) \oplus i Hol_{c}(\Omega)
\end{equation}
and, analogously,
\begin{equation}\label{Bergman-direct-sum}
\mathcal A^2(\Omega)=\mathcal A_{c}^2(\Omega)\oplus i\mathcal A_{c}^2(\Omega).
\end{equation}
\end{cor}
\begin{proof} Set
 $$f_1 = \frac{1}{2}(f+ Z_{\mathbb C}\circ f\circ Z_{\mathbb C}), \quad f_2= \frac{i}{2}(-f+ Z_{\mathbb C}\circ f\circ Z_{\mathbb C}).$$
\end{proof}

We now list some properties of the elements of $\mathcal A_c^2(\Omega)$.
\begin{pro}\label{pro3}
Let $\Omega\subseteq\mathbb{C}$ be a Z-invariant open set.
 \begin{enumerate}
   \item   Let  $f\in \mathcal A_c^2(\Omega)$, then  $\displaystyle \int_{\Omega} f(\zeta)d\sigma_{\zeta} \in \mathbb R $.
   \item  Let  $f\in \mathcal A^2(\Omega)$   and let $f_1,f_2\in \mathcal A_c^2(\Omega)$ such that
       $f=f_1+ i f_2$. Then
       \begin{enumerate}
         \item $\displaystyle \int_{\Omega} f_1(\zeta)d\sigma_{\zeta}=\int_{\Omega} Re \; f(\zeta)d\sigma_{\zeta}$ {} and {}
         $\displaystyle \int_{\Omega} f_2(\zeta)d\sigma_{\zeta}=\int_{\Omega} Im \; f(\zeta)d\sigma_{\zeta},$

         where $Re \; f$ and $Im \; f$ are, respectively, the real and the imaginary parts of $f$.

         \item $\displaystyle\left. \begin{array}{l} \|f_1\|_{\mathcal A^2(\Omega)}, \\\|f_2\|_{\mathcal A^2(\Omega)}\end{array} \right\}\leq \|f\|_{\mathcal A^2(\Omega)}\leq \|f_1\|_{\mathcal A^2(\Omega)}+ \|f_2\|_{\mathcal A^2(\Omega)}$.
                 \end{enumerate}

                 \item  Let $f,g\in \mathcal A_c^2(\Omega)$ then
                 $\langle g,f \rangle_{\mathcal A^2(\Omega)} \in \mathbb R$, and let $h\in  \mathcal A_{\bar c}^2(\Omega)$ then
                 $\langle h,f \rangle_{\mathcal A^2(\Omega)}$ is a pure imaginary complex number.
              \end{enumerate}
\end{pro}
\begin{proof} The results are obtained by direct calculations and using the C-property.
\end{proof}

\begin{ob} Let $\Omega\subseteq\mathbb{C}$ be a Z-invariant open set. Define  the  {operation}
$$
+ : \left(\mathcal A_c^2(\Omega)\times \mathcal A_c^2(\Omega)\right) \times \left(\mathcal A_c^2(\Omega)\times \mathcal A_c^2(\Omega)\right) \to \mathcal A_c^2(\Omega)\times \mathcal A_c^2(\Omega)
$$
as follows:
  $$(f_1 \ , \ f_2)+ (g_1 \ , \ g_2)= (f_1 + g_1  \ , \   f_2 +g_2),$$
  and  the following product of a pair with a scalar $$(f_1 \ , \ f_2)\lambda=(f_1\lambda \ , \ f_2\lambda)$$ for any $\lambda\in \mathbb R$.
We obtain an  {isomorphism} between the  real-linear spaces $\mathcal A^2(\Omega) $ and  $ \mathcal A_c^2(\Omega)\times \mathcal A_c^2(\Omega)$.
\end{ob}

\begin{pro}\label{pro2}
Let $\Omega\subseteq\mathbb{C}$ be a Z-invariant open set.
  The   spaces $ \mathcal A_{c}^2(\Omega)$ and $ \mathcal A_{\bar c}^2(\Omega)$ are real Hilbert spaces.  \end{pro}
  \begin{proof}
   We prove that the space $ \mathcal A_{c}^2(\Omega)$ is complete, since the completeness of $ \mathcal A_{\bar c}^2(\Omega)$ follows from the map $f\mapsto if$ and  the result for  $ \mathcal A_{c}^2(\Omega)$.
   \\
   Let $\{f_n\}_{n\in\mathbb N}\subset \mathcal A_c^2(\Omega)$ be a Cauchy sequence, then there exists $f \in \mathcal A^2(\Omega)$  such that $\{f_n\}_{n\in\mathbb N}$ converges to $f$ in norm.   Lemma 1.4.1 of \cite{K}, implies that the convergence in the norm of $\mathcal A^2(\Omega)$ implies the    uniformly convergence on compact subsets of $\Omega$. Therefore
  $$Z_{\mathbb C}\circ f\circ Z_{\mathbb C} (z)=\lim_{n\to \infty} Z_{\mathbb C}\circ f_n\circ Z_{\mathbb C} (z)=\lim_{n\to \infty} f_n (z)=f (z), \quad \forall z\in \Omega.$$
  Then $f\in  \mathcal A_{c}^2(\Omega)$.
\end{proof}

\begin{nton} Let $\Omega\subseteq\mathbb{C}$ be a Z-invariant open set.
Let  $\mathcal K_{\Omega}(\cdot ,\cdot )$ be the Bergman kernel associated to
 $\Omega$, and let $\mathcal R_{\Omega}(\cdot ,\cdot )$ and $\mathcal I_{\Omega}(\cdot ,\cdot )$ be  the functions defined on $\Omega\times \Omega$ such that for any $z\in \Omega$, one has   $\mathcal R_{\Omega}(z, \cdot),\mathcal I_{\Omega}(z,\cdot)\in \mathcal A_c^2(\Omega)$ and
  $$\mathcal K_{\Omega}(z, \cdot)=\mathcal R_{\Omega}(z, \cdot)+ i \mathcal I_{\Omega}(z,\cdot).$$
\end{nton}

The two functions $\mathcal R_{\Omega}$, $\mathcal I_{\Omega}$ allow to obtain the real and imaginary part
of a function $f\in \mathcal A_c(\Omega)$ as described in the following result:

\begin{teo}\label{pro13}
Let $\Omega$ be a $Z$-invariant open set in $\mathbb{C}$, and $f\in \mathcal{A}_c(\Omega)$, then
$$
\displaystyle Re  \ f(z)=\int_{\Omega}\mathcal R_{\Omega}(z, \zeta)f(\zeta)d\sigma_{\zeta} ,
$$
and
$$
Im  \ f(z)=\int_{\Omega} \mathcal I_{\Omega}(z,\zeta)f(\zeta)d\sigma_{\zeta},
$$
for all $z\in \Omega$.
\end{teo}
\begin{proof}
Let us denote by $\phi_1,\phi_2$ the two projections from $\mathbb C$ to $\mathbb R$ given by
$$\phi_1(z)=Re \; z ,\ \ \textrm{and} \ \ \phi_2(z)=Im \; z,\quad \forall z\in \mathbb C.$$
Since the evaluation functional  $\phi_z$ is bounded on $\mathcal A^2(\Omega)$, also  the  real-linear functionals:  $\alpha_z:=\phi_1 \circ \phi_z$,  and $\beta_z:=\phi_2 \circ \phi_z$  are  bounded on $\mathcal A_c^2(\Omega)$. Note that
$$
\alpha_z[f] = Re \ f(z) , \quad \textrm{and}    \quad  \beta_z[f](z)= Im\ f(z),\quad \forall f\in \mathcal A_c^2(\Omega).
$$

The Riesz's representation theorem for real-linear  Hilbert spaces implies the existence of
  $R_z, \ I_z \in \mathcal A_c^2(\Omega)$, such that
  \begin{equation}\label{Ref}
  Re  \, f(z)=\int_{\Omega}\overline{ R_z(\zeta)}f(\zeta)d\sigma_{\zeta},\quad \quad \forall f\in \mathcal A_c^2(\Omega)
  \end{equation}
  and
   \begin{equation}\label{Imf}
   Im  \, f(z)=\int_{\Omega}\overline{ I_z(\zeta)}f(\zeta)d\sigma_{\zeta}, \quad \quad \forall f\in \mathcal A_c^2(\Omega).
   \end{equation}

  Denoting $\dot{ R}(z,\cdot)=\overline{R_z }$ and $\dot{I}(z,\cdot)=\overline{I_z }$
   we can rewrite (\ref{Ref}) and (\ref{Imf}) as
$$
 Re  \, f(z)=\int_{\Omega} \dot{ R}(z,\zeta)f(\zeta)d\sigma_{\zeta},\quad \textrm{and}\quad
Im  \, f(z)=\int_{\Omega} \dot{ I}(z,\zeta)f(\zeta)d\sigma_{\zeta},\quad \forall f\in \mathcal A_c^2(\Omega).
$$

Now, defining   ${\bf{K}}_{\Omega}(z,\cdot):=\dot{ R}(z,\cdot) + i \dot{ I}(z,\cdot)$ one has
$$\displaystyle f(z)= \int_{\Omega}{\bf{K}}_{\Omega}(z,\zeta)f(\zeta) d\sigma_{\zeta} ,\quad \forall f\in \mathcal A^2(\Omega).$$
Since $R_w$, $I_w$ belong to $\mathcal{A}^2_c(\Omega)$ we have that $\overline{{\bf{K}}_{\Omega} (w, \cdot)}=  R_w -i   I_w \in \mathcal A^2(\Omega)$, for any $w\in\Omega$ and
$$
\overline{{\bf{K}}_{\Omega} (w,z) } = \int_{\Omega}{\bf{K}}_{\Omega} (z,\zeta) \overline{{\bf{K}}_{\Omega} (w,\zeta)} d\sigma_{\zeta},
$$
Applying the complex conjugation on the previous expression one gets
$${\bf{K}}_{\Omega} (w,z)  = \int_{\Omega}{\bf{K}}_{\Omega} (w,\zeta) \overline{{\bf{K}}_{\Omega} (z,\zeta)} d\sigma_{\zeta}=\overline{{\bf{K}}_{\Omega} (z,w) }, $$
which  means  that the function ${\bf{K}}_{\Omega} $ is hermitian.
The uniqueness of the Bergman kernel implies that ${\bf{K}}_{\Omega} =\mathcal K_{\Omega}$. Therefore
$\mathcal R_{\Omega}=\dot{ R}$ and $\mathcal I_{\Omega} =\dot{I}$ by Corollary \ref{pro0}.
\end{proof}

\begin{ob}{\rm
Note that $\alpha_{z}[f]=\alpha_{\bar z}[f],$ and  $\beta_{z}[f]=- \beta_{\bar z}[f]$,        {for all} $ f\in \mathcal A_c^2(\Omega)$.
}
\end{ob}
From now on, for the sake of simplicity, we will omit the subscript "$\Omega$" when referring to $\mathcal{R}$, $\mathcal{I}$.
\begin{pro}\label{cor1}(Some properties of $ \mathcal R$ and $ \mathcal I$)
Let $\Omega$ be a $Z$-invariant domain in $\mathbb{C}$.
\begin{enumerate}
  \item  $\mathcal R(z,w)=\overline{\mathcal R( z,\bar{w})}$ {} and {} $\mathcal I(z,w)=\overline{\mathcal I( z,\bar{w})}$, for all $w\in \Omega$.
        \item $\mathcal R(z,\bar z) - i\mathcal I(z,z) =\mathcal R(z,z) +i \mathcal I(z,\bar z)$.
         \item $\displaystyle \mathcal R(z,\bar z) -i \mathcal I(z,z)= \int_{\Omega}|\mathcal R(z,\zeta)|^2- |\mathcal I(z,\zeta)|^2 d\sigma_{\zeta}$.
\end{enumerate}

\end{pro}
\begin{proof}

\begin{enumerate}
\item  We have $\mathcal R(z,w)=\overline{ R_z(w)}=R_z(\bar w)=\overline{\mathcal R( z,\bar{w})}$. The proof for the function $\mathcal I$ is similar.

 \item The property follows from  the fact that all the functions  $f\in \mathcal A_c^2(\Omega)$ satisfy
  $Re  \ \bar f(z)= Re  \ f(\bar z)$ and  $ -Im \  f(z)=Im \bar  f(z)=Im f(\bar  z)$.

\item It is a  direct consequence of the previous facts.
\end{enumerate}
\end{proof}

The version of  Theorem \ref{pro13} which holds for functions in $ \mathcal A_{\bar c}^2(\Omega)$ is given in the next result.
\begin{pro}\label{pro14}
Let $\Omega\subseteq\mathbb C$ be a Z-invariant open set.
Given   $f\in \mathcal A_{\bar c}^2(\Omega)$,  one has
  $$
  i  Im \, f(z)=\int_{\Omega} \mathcal R(z,\zeta)f(\zeta)d\sigma_{\zeta},
  $$
  and
  $$
   -i Re \, f(z) = \int_{\Omega} \mathcal I(z,\zeta) f(\zeta) d\sigma_{\zeta}.
   $$
\end{pro}
\begin{proof}
Note that   $i f \in \mathcal A_c^2(\Omega)$,  then
$$Re \, if (z)=\int_{\Omega} \mathcal R(\zeta, z) i f(\zeta) d\sigma_{\zeta},$$
or equivalently
 $$i Im \, f (z)= \int_{\Omega} \mathcal R(\zeta, z)  f(\zeta) d\sigma_{\zeta}.$$
For the function $\mathcal I$ the proof is similar.
\end{proof}

\begin{cor}\label{cor15}
Let $\Omega\subseteq\mathbb C$ be a Z-invariant open set.
  Given any $f\in \mathcal A^2(\Omega)$, let $f_1\in \mathcal A_c^2(\Omega)$ and $f_2\in $ $\mathcal A_{\bar c}^2(\Omega)$ be such that
  $f=f_1+ f_2$. Then
  $$\int_{\Omega} \mathcal R(\zeta, z) f(\zeta )d\sigma_{\zeta} = Re \ f_1(z) + i Im \ f_2(z),$$
  and
  $$  \int_{\Omega} \mathcal I(\zeta, z) f(\zeta )d\sigma_{\zeta} = Im \ f_1(z) - i Re \ f_2(z).$$
\end{cor}
\begin{proof}
It is direct consequence of  Theorem \ref{pro13}, and  Proposition \ref{pro14}.
\end{proof}

\begin{ob}
{\rm Let us denote by ${\bf B}_\Omega:\ \mathcal{L}^2(\Omega)\to \mathcal{A}^2(\Omega)$ the Bergman projection associated to $\Omega$:
 $$
 {\bf B}_\Omega\ : f\ \mapsto \int_\Omega \mathcal{K}_\Omega (\cdot, \zeta)f(\zeta)d\sigma_{\zeta}.
 $$
 The previous results allow to see that the Bergman projection is the sum of two operators:
$$
{\bf R}_{\Omega}[f](z)=  \int_{\Omega} \mathcal R(\zeta, z) f(\zeta )d\sigma_{\zeta} ,\quad {\bf I}_{\Omega}[f](z)=  \int_{\Omega} \mathcal I(\zeta, z) f(\zeta )d\sigma_{\zeta},\quad \forall f\in \mathcal A^2(\Omega).
$$
In other words,   $\mathbf{B}_{\Omega} = \mathbf{R}_{\Omega} + i \mathbf{I}_{\Omega}$, and the behavior of these new operators is given in Theorem \ref{pro13}, Proposition \ref{pro14} and Corollary \ref{cor15}.
}
\end{ob}

\begin{pro} \ {}
Let $\Omega\subseteq\mathbb C$ be a Z-invariant open set. Then the functions $\mathcal R$, $\mathcal I$ satisfy:
  \begin{enumerate}
    \item $\mathcal I(\bar z, \zeta)=- \mathcal I(z,\zeta)$ {} and {} $\mathcal R(\bar z, \zeta)=\mathcal R(z,\zeta)$.

    \item  $\mathcal I(\bar z, \bar \zeta)= - \overline{\mathcal I(z,\zeta)}$ {} and {} $\mathcal R(\bar z, \bar   \zeta)=\overline{\mathcal R(z,\zeta)}$.
      \end{enumerate}
\end{pro}
\begin{proof}
\begin{enumerate}
  \item From the fact $-Im \ f(z)= Im \overline{f(z)}=Im \ f(\bar z)$, for all $f\in \mathcal A_c^2(\Omega)$, one  {obtains} that
$$-\int_{\Omega} \mathcal I(z, \zeta) f(\zeta) d\sigma_{\zeta}=\int_{\Omega} \mathcal I(\bar z, \zeta) f(\zeta) d\sigma_{\zeta}, \quad \forall f\in \mathcal A_c^2(\Omega).$$
Therefore
$$\int_{\Omega}\left( \mathcal I(z, \zeta)+  \mathcal I(\bar z, \zeta)\right) f(\zeta) d\sigma_{\zeta}=0, \quad \forall f\in \mathcal A_c^2(\Omega).$$
In particular, $\overline{ \mathcal I(z, \cdot )+  \mathcal I(\bar z, \cdot)}\in \mathcal A_c^2(\Omega)$, then $$\| \mathcal I(z, \cdot)+  \mathcal I(\bar z, \cdot)\|_{\mathcal A^2(\Omega)}=0.$$
The proof of the fact  $\mathcal R(\bar z, \zeta)=\mathcal R(z,\zeta)$ is similar to the previous reasoning.
\item Use  the previous case and Corollary \ref{cor1}.
\end{enumerate}
\end{proof}

\section{Slice regular functions and the C-property}

We consider the space  $\mathbb R^3$ embedded in  $\mathbb H$ as follows $$(a_1,a_2,a_3) \mapsto a_1e_1+a_2e_2+a_3e_3,$$
where $\{e_0=1,e_1,e_2,e_3\}$ is the usual basis of the quaternions.
    Let $\mathbb S^2$ be the sphere of purely imaginary unit quaternions and let $\bf i\in\mathbb S^2$. The real space generated by $\{1, {\bf i}\}$, denoted by $\mathbb C({\bf i})$,   is isomorphic, not only as a  {linear} space but even as a field,  to  the field of the complex numbers.   Given a domain $\Lambda\subset\mathbb H$, let  $\Lambda_{\bf i}=\Lambda\cap \mathbb C({\bf i})$  and   $Hol( \Lambda_{\bf i})$ represents the  complex linear space of holomorphic functions from   $\Lambda_{\bf i}$  to $\mathbb C({\bf i})$. \\
    Any nonreal quaternion $q=q_0+e_1q_1+e_2q_2+e_3q_3$ can be uniquely written in the form $q=x+I_qy$ where
     $x=q_0$,  $\displaystyle I_q=\frac{\vec q}{\|\vec q\|}\in\mathbb{S}^2$, and $y=\|\vec q\|$ thus it  belongs to the complex plane $\mathbb C(I_q)$.
\\
\begin{defin}
A real differentiable quaternionic-valued function $f$ defined on an open set $\Lambda \subset\mathbb{H}$ is called (left) slice regular on $\Lambda$ if, for any ${\bf i}\in \mathbb S^2$, the function $f_{\mid_{\Lambda_{{\bf i}}}}$
 is such that
$$
 \left( \frac{\partial}{\partial x} + {\bf i}  \frac{\partial}{\partial y}\right) f_{\mid_{\Lambda_{{\bf i}}}}(x+y{\bf i})=0, \textrm{ on $\Lambda_{{\bf i}}$}.
 $$
The function $f$ is called anti-slice regular on the right if, for any ${\bf i}\in \mathbb S^2$, the function $f_{\mid_{\Lambda_{{\bf i}}}}$  is such that
$$
 \frac{\partial  }{\partial x}f_{\mid_{\Lambda_{{\bf i}}}}(x+y{\bf i})- \frac{\partial}{\partial y}f_{\mid_{\Lambda_{{\bf i}}}}(x+y{\bf i}){\bf{i}}=0, \textrm{ on $\Lambda_{{\bf i}}$}.
$$
We denote by $\mathcal{SR}(\Lambda)$ the set of slice regular functions on $\Omega$.
\end{defin}
The theory of slice regular functions is meaningful only if the open sets on which they are defined have suitable properties, thus we introduce the following definition.
\begin{defin}
Let $\Lambda\subseteq\mathbb{H}$ we say that $\Lambda$ is axially symmetric if whenever $q=x+I_qy$ belongs to $\Lambda$ all the elements $x+Iy$ belong to $\Lambda$ for all $I\in\mathbb{S}^2$. We say that $\Lambda$ is a slice domain, or s-domain for short, if it is a domain intersecting the real axis and such that $\Lambda\cap \mathbb C({\bf i})$ is connected for all ${\bf i}\in\mathbb{S}^2$.
\end{defin}

In this section,  ${\bf i, \bf j} \in\mathbb S^2$ are mutually orthogonal vectors,  and   $\Lambda\subset \mathbb H$  will be an axially symmetric s-domain. Therefore for any $\bf i$, the domain $\Lambda_{\bf i}$ is $Z$-invariant in the complex plane $\mathbb C({\bf i})$.

The Splitting Lemma and the Representation Formula, see \cite{csTrends}, imply the good definition of the following operators, which  relate the slice regular space with the space of pairs of holomorphic functions on $\Lambda_{\bf i}$:
  $$
  \begin{array}{lrcl}
    Q_{\bf i} : &  \mathcal{SR}(\Lambda)  &\longrightarrow &  Hol(\Lambda_{\bf i})+ Hol(\Lambda_{\bf i}){\bf j}\\
 &  & & \\
Q_{\bf i} :  & f & \longmapsto & f\mid_{\Lambda_{\bf i}},\quad \quad  \end{array}
  $$
and
   $$ P_{\bf i} :   Hol(\Lambda_{\bf i})+ Hol(\Lambda_{\bf i}){\bf j} \longrightarrow  \mathcal{SR}(\Lambda),$$ defined for any $f\in  Hol(\Lambda_{\bf i})+ Hol(\Lambda_{\bf i}){\bf j}$  by
   $$
   P_{\bf i}[f](q)=P_{\bf i}[f](x+yI_q)=\frac{1}{2}\left[(1+ I_q{\bf i})f(x-y{\bf i}) + (1- I_q {\bf i}) f(x+y {\bf i})\right].
   $$
 Moreover, we have  that
$$P_{\bf i}\circ Q_{\bf i}= I_{\mathcal{SR}(\Lambda)} \quad \textrm{and}
 \quad  Q_{\bf i}\circ P_{\bf i}= I_{ Hol(\Lambda_{\bf i})+ Hol(\Lambda_{\bf i}){\bf j} } $$
where $I$ denotes the identity operator.

In the sequel, we will use the notations for holomorphic functions with the C- and the anti-C-property, presented  in the previous section.

\begin{pro}\label{pro1} \ {}
Let $\Lambda$ be an axially symmetric s-domain and let $\Lambda_{\bf i}=\Lambda \cap \mathbb{C}_{\bf i}$.
\begin{enumerate}
\item If   $f\in Hol_c(\Lambda_{\bf i}) $,  then
$$P_{\bf i}[f](q)=P_{\bf i}[f](x+y I_q)= Re \, f(x+y {\bf i}) +  I_q Im \, f(x+y {\bf i}),\quad \forall q\in\Lambda.$$
\item  If $f\in Hol_{\bar c}(\Lambda_{\bf i}) $, then
   $$
   P_{\bf i}[f](q)= \left( Im  \, f(x+y {\bf i}) -  I_q Re \, f(x+y {\bf i}) \right){\bf i},\quad \forall q\in\Lambda,
   $$
\end{enumerate}
where, for both cases, $x=q_0$,  $y=\|\vec q\|$ and the  real-valued functions  $Re \, f$ and $Im \, f$ are such that $f(x+{\bf i}y)= Re \, f(x+{\bf i}y) + {\bf i} Im \, f(x+{\bf i}y)$.
\end{pro}
\begin{proof}
It is a direct consequence of the C-property, anti-C-property and the Representation Formula,  \cite{csTrends}.
\end{proof}

\begin{pro} Let $f\in Hol(\Lambda_{\mathbf{i}})$. Then $P_{\bf i}[f]$ is
given by
 $$P_{\bf i}[f](q)=Re \, f (x+y{\bf i}) + I_q  Im \, f(x+y{\bf i}) +  (1+I_q{\bf i}) f_2(x-y {\bf i}),$$
 where $f_2$ is the part of the function $f$ which satisfies the anti-C-property.
\end{pro}
\begin{proof}  Let $f_1, f_2$ be the parts of $f$ which  {satisfy}, respectively, the C-property and the anti-C-property.
Then
$$\begin{array}{ll}
  P_{\bf i}[f](q)& =P_{\bf i}[f_1](q)+ P_{\bf i}[f_2](q)\\
  &\\
  & = Re \ f_1(x+y{\bf i}) +  I_q Im \, f_1(x+y{\bf i})  +  Im \, f_2(x+y {\bf i}){\bf i} - Re \, f_2(x+y{\bf i}) \  \ I_q {\bf i}\\
  &\\
  &= Re \, f (x+y{\bf i}) + I_q  Im \, f(x+y{\bf i})   +   (1+I_q{\bf i}) f_2(x-y {\bf i}).
\end{array}$$
\end{proof}

In the sequel, the quaternionic conjugation will be denoted  by $ Z_{\mathbb H}(q)=\bar q$.
\begin{defin}
Let $\Lambda$ be an axially symmetric open set in $\mathbb H$.
\begin{enumerate}
\item
Consider a function $f\in \mathcal{SR}(\Lambda)$, we say that $f$ satisfies that quaternionic C-property  if and only if $f=Z_{\mathbb H }\circ f\circ Z_{\mathbb H}$.
We say that $f$ satisfies the anti-C-property if and only if $f=-Z_{\mathbb H }\circ f\circ Z_{\mathbb H}$.
\item The subset of $\mathcal{SR}(\Lambda)$ of functions satisfying the C-property (resp. the anti-C-property) will be denoted by
 $\mathcal{SR}_c(\Lambda)$  (resp. $\mathcal{SR}_{\bar c}(\Lambda)$).
\end{enumerate}
\end{defin}

\begin{ob}
 {\rm
 Note that there exist slice regular functions which  do not satisfy the quaternionic C-property or anti-C-property. For example, given a quaternion $\lambda \in \mathbb H\setminus \mathbb R$ the function $f(q)= q\lambda $ is a slice regular function but  $f$ does not belong to $\mathcal{SR}_c(\Lambda)$.
}
\end{ob}
\begin{ob}
 {\rm
It is easy to verify that a slice regular function which admits power series expansion at a real point belongs to $\mathcal{SR}_c(\Lambda)$ if and only if the coefficients of the power series are real numbers. Similarly,  a slice regular function which admits power series expansion at a real point belongs to $\mathcal{SR}_{\bar c}(\Lambda)$ if and only if the coefficients of the power series are purely imaginary quaternions. In other papers, sometimes the set $\mathcal{SR}_c(\Lambda)$ is denoted by $\mathcal{N}(\Lambda)$ and  the functions belonging to it are called real slice regular functions (see \cite{gp}).
}
\end{ob}

\begin{pro}\label{cor3}
Let $\Lambda$ be an axially symmetric s-domain and recall that $\Lambda_{{\bf i}}=\Lambda\cap\mathbb{C}_{{\bf i}}$.

     \begin{enumerate}
       \item If $f\in Hol_c(\Lambda_{\bf i})$, then $P_{\bf i}[f]\in\mathcal{SR}_{ c}(\Lambda)$.
       \item If $f\in Hol_{\bar c}(\Lambda_{\bf i})$, then  $P_{\bf i}[f]\in\mathcal{SR}_{\bar c}(\Lambda)$.
             \end{enumerate}
\end{pro}
\begin{proof}
\begin{enumerate}
  \item  Let $q\in \Lambda$, then Proposition \ref{pro1} implies that
$$P_{\bf i}[f](q)= Re \, f(x+y {\bf i}) + I_q Im \, f(x+y {\bf i}).$$
Now, we see that
$$  \overline{P_{\bf i}[f]( q)}  =Re \, f(x+y {\bf i}) +  (- I_q)  Im \, f(x+y {\bf i})=
  P_{\bf i}[f](x+y(-I_q)) =  P_{\bf i}[f](\bar q).$$

  \item  {It is based on the observation that if $f\in Hol_{\bar c}(\Lambda_{\bf i})$ then ${\bf i} f$ satisfies the C-property.  By the previous step we have}
   $$
   Z_{\mathbb H}\circ P_{\bf i}[ {\bf i} f] \circ Z_{\mathbb H}= P_{\bf i}[{\bf i}f].
   $$
We finish the  proof   using the following facts:
   $$
   P_{\bf i}[{\bf i}f]=P_{\bf i}[f]{\bf i},
   $$
   and
   $$
   Z_{\mathbb H}\circ P_{\bf i}[ {\bf i} f] \circ Z_{\mathbb H}=-Z_{\mathbb H}\circ P_{\bf i}[ f] \circ Z_{\mathbb H}{\bf i}.
   $$
   \end{enumerate}
\end{proof}

The following result characterizes the image of  $Hol_{c}(\Lambda_{\bf i})$ through $P_{\bf i}$:
\begin{pro}\label{pro18}
Let $\Lambda\subseteq\mathbb H$ be an axially symmetric open set. Then
$$
\mathcal{SR}_c(\Lambda)= P_{\bf i} [Hol_{c}(\Lambda_{\bf i})].
$$
\end{pro}
\begin{proof}
Proposition \ref{cor3} shows that $P_{\bf i} [Hol_{c}(\Lambda_{\bf i})]\subseteq \mathcal{SR}_c(\Lambda)$. Now suppose
that
$f\in \mathcal{SR}_c(\Lambda) $, i.e.,
$f(q)=\overline{ f(\bar q)}$,  $\forall q\in \Lambda$. In particular,  for any $q\in \Lambda_{\bf i}$, one has
\begin{equation}\label{equ16}
Q_{\bf i }[ f](q)=\overline{Q_{\bf i} [f](\bar q)}.
\end{equation}
On the  other hand, the Splitting Lemma implies that for any choice of ${\bf i}, {\bf j}\in\mathbb{S}^2$ with ${\bf i}\perp {\bf j}$ there exist  $f_1,f_2\in Hol(\Lambda_{\bf i})$ such that
$Q_{\bf  i}[f]=f_1+f_2{\bf j}$. From  (\ref{hol-direct-sum})  there exist $h_0, h_1, h_2,h_3\in Hol_{c}(\Lambda_{\bf i})$, such that $f_1=h_0+h_1{\bf i}$ and $f_2=h_2+h_3{\bf i}$.
Then we get
$$
Q_{\bf  i}[f]=h_0+h_1{\bf i}  +  h_2 {\bf j}  +h_3{\bf i} {\bf j},
$$
and (\ref{equ16}) implies that
$$
h_0(q)+h_1(q){\bf i}  +  h_2(q) {\bf j}  +h_3(q){\bf i}{\bf j}   =  \overline{ h_0(\bar q)} -  {\bf i}\overline{h_1(\bar q)}  -  {\bf j} \overline{h_2(\bar q)} - {\bf i} {\bf j}  \overline{h_3(\bar q) } ,  \quad \forall q\in \Lambda_{\bf i}.
$$
Using the C-property of each  function $h_\ell$, $\ell=0,\ldots, 3$, one has that $h_1=h_2=h_3=0$. Then $f=P_{\bf i}[h_0]$ and so $\mathcal{SR}_c(\Lambda)=P_{\bf i}[Hol_c(\Lambda_{\bf i})]$.
\end{proof}

The proof of the previous result allows to refine the Splitting Lemma,  which, as we recalled above, establishes the existence of two holomorphic functions such that $Q_{\bf  i}[f]=f_1+f_2{\bf j}$. In fact we have:
\begin{cor} (Refined Splitting Lemma)
Let $\Lambda$ be an axially symmetric open set, $f\in\mathcal{SR}(\Lambda)$ and ${\bf i}, {\bf j}\in\mathbb{S}^2$ with ${\bf i}\perp {\bf j}$. Then there exist four functions $h_\ell\in  Hol_{c}(\Lambda_{\bf i})$, $\ell=0,\ldots, 3$ such that
$$
Q_{\bf i}[f]=h_0+h_1{\bf i}  +  h_2 {\bf j}  +h_3{\bf i} {\bf j}.
$$
\end{cor}
 {This result appeared first in Remark 3 in \cite{gp}, in a more general setting, but it has been called Refined Splitting Lemma in \cite{css}.}
Another useful consequence is:
\begin{cor}\label{comm}
For any  $f,g \in \mathcal{SR}_c(\Lambda) $, one has $fg=gf$.
\end{cor}
\begin{proof}
Proposition  \ref{pro18}   implies that there exist $h_1,h_2\in Hol_c(\Lambda_{\bf i})$ such that $f=P_{\bf i}[h_1]$ and $g=P_{\bf i}[h_2]$, and Proposition   \ref{pro1} says that for each $q\in \Lambda$ the values $f(q)=P_{\bf i}[h_1](q)$ and
$g(q)=P_{\bf i}[h_2](q)$ belong  to the same slice, specifically to the slice generated by $\{1, {\bf i}_q \}$.
\end{proof}
The  next result generalizes property (\ref{hol-direct-sum}) in Corollary \ref{pro0} to the case of slice regular functions, providing a decomposition of $\mathcal{SR}(\Lambda)$.
\begin{pro}\label{pro20}
Let $\Lambda$ be an axially symmetric s-domain and let $\{1,{\bf i},{\bf j},{\bf i}{\bf j}\}$ be a basis of $\mathbb{H}$. Then:
$$\mathcal{SR}(\Lambda)=\mathcal{SR}_c(\Lambda)  \oplus \mathcal{SR}_c(\Lambda){\bf i}   \oplus  \mathcal{SR}_c(\Lambda)  {\bf j}\oplus \mathcal{SR}_c(\Lambda) \bf{ij}.
$$
\end{pro}
\begin{proof}
Let   $f\in \mathcal{SR}(\Lambda)$, then	the Splitting Lemma implies the existence of  $f_1,f_2\in Hol(\Lambda_{\bf i})$ such that
$Q_{\bf  i}[f]=f_1+f_2{\bf j}$. By formula (\ref{hol-direct-sum}),  there exist $h_0, h_1, h_2,h_3\in Hol_{c}(\Lambda_{\bf i})$, such that $f_1=h_0+h_1{\bf i}$ and $f_2=h_2+h_3{\bf i}$.
Thus we can write
$$Q_{\bf  i}[f]=h_0+h_1{\bf i}  +  h_2 {\bf j}  +h_3{\bf i} {\bf j},$$
and
$$f= P_{\bf i}[h_0]+ P_{\bf i}[h_1]+ P_{\bf i}[h_2]+ P_{\bf i}[h_3].$$
Proposition \ref{pro18}  implies that
$\mathcal{SR}(\Lambda)=\mathcal{SR}_c(\Lambda)  + \mathcal{SR}_c(\Lambda){\bf i}   + \mathcal{SR}_c(\Lambda)  {\bf j}+ \mathcal{SR}_c(\Lambda) \bf{ij}$.

To show that the sum is a direct sum,
suppose that   $f\in \mathcal{SR}_c(\Lambda)\cap  \mathcal{SR}_c(\Lambda){\bf i} $. Then there exists $g\in  \mathcal{SR}_c(\Lambda)$, such that   $f=g{\bf i}$. From Proposition \ref{pro18}, there exist $h_1,h_2\in Hol_c(\Lambda_{\bf i})$ such that
$f=P_{\bf i}[h_1]$, and $g=P_{\bf i}[h_2]$. Then $h_1 =h_2{\bf i}$, and for any $q\in \Lambda_{\bf i}$ one has
$$h_2(q){\bf i}=h_1(q)=\overline{h_1(\bar q)} = \overline{h_2(\bar q){\bf i}} =-  h_2(q){\bf i} , $$
then $h_2=h_1=0$, and $\mathcal{SR}_c(\Lambda)\cap  \mathcal{SR}_c(\Lambda){\bf i} =\{0\}$.

Similarly one can see that all the other intersections between
$\mathcal{SR}_c(\Lambda)$,  $\mathcal{SR}_c(\Lambda){\bf i}$, $\mathcal{SR}_c(\Lambda){\bf j}$, $\mathcal{SR}_c(\Lambda){\bf k}$, are $\{0\}$ and the statement follows.
\end{proof}
The Representation Formula shows that all the slice regular functions defined on axially symmetric s-domains are of the  form (see \cite{csTrends})
\begin{equation}\label{Hfunctions}
f(q)=f(x+I_qy)=\alpha(x,y)+I_q\beta(x,y)
\end{equation}
where  {$\alpha$, $\beta$ are $\mathbb{H}$-valued,} such that
\begin{equation}\label{e-o}
\alpha(x,-y)=\alpha(x,y),\\
\beta(x,-y)=-\beta(x,y)
\end{equation}
and they satisfying the Cauchy-Riemann system
\begin{equation}\label{CR}
\frac{\partial}{\partial x}\alpha - \frac{\partial}{\partial y}\beta =0, \ \ \ \ \ \ \ \ \
\frac{\partial}{\partial y}\alpha + \frac{\partial}{\partial x}\beta =0.
\end{equation}
However, one may consider functions of the form (\ref{Hfunctions}) satisfying (\ref{e-o}) and (\ref{CR})  defined on axially symmetric open sets $\Lambda$.
 {These functions correspond to the slice regular functions quaternion valued in the terminology of \cite{gp}, in which these functions are studied in a more general setting.}
\begin{defin}\label{accalamb}
We denoted by $\mathcal{H}(\Lambda)$ the set of  functions  of the form
(\ref{Hfunctions}) satisfying (\ref{e-o}) and (\ref{CR}) defined on axially symmetric open sets $\Lambda$.
\end{defin}

  When the axially symmetric set $\Lambda$ is, in particular, an s-domain then $\mathcal{H}(\Lambda)=\mathcal{SR}(\Lambda)$.
\begin{pro}
Let $f\in \mathcal H(\Lambda)$  then consider    $\alpha$, $\beta$ given in Definition \ref{accalamb}. Then $f $ satisfies the quaternionic C-property if and only if the functions $\alpha$ and $\beta$ are real valued.
\end{pro}
\begin{proof}
If the functions $\alpha$ and $\beta$ are real valued function, it is immediate that $f$ satisfies the C-property.
Now, suppose that $f$ satisfies the C-property. Since
$$\alpha(x,y)  = \frac{1 }{2}   \left[  f(x+{\bf i}  y)   +   f(x-{\bf i}  y)    \right] $$
applying the quaternionic conjugation one has that
 {$$\overline{\alpha(x,y)}  =  \frac{1 }{2}   \left[ \overline{ f(x+{\bf i}  y) }  +   \overline{f(x-{\bf i}  y) }   \right]    =  \frac{1 }{2}   \left[ \overline{ f(x-{\bf i}  y) }  +   \overline{f(x+{\bf i}  y) }   \right] = \alpha(x,y),$$}
 That means $\alpha$  is a real valued function.
\\
Then  the function $f(x+{\bf i} y) =\alpha(x,y) + {\bf i}  \beta(x,y) $ satisfies the C-property and $\alpha$ is a real valued function. Therefore  $- \overline{\beta(x,y)} \ {\bf i} = - {\bf i} \beta(x,y) $,  or equivalently
$ 0=  \overline {\beta(x,y)} {\bf i} + \bar{\bf i} \beta(x,y)=2 \langle {\bf i} , \beta(x,y) \rangle_{\mathbb R^4}$, as $\beta$ does not depend of the unit vector ${\bf i}$ the previous identity  is  for all ${\bf i} \in \mathbb S^2$. Then  $\beta$ is a real valued function.
\end{proof}
The refined Splitting Lemma and Proposition \ref{pro20} hold also in this setting. Denote by  $\mathcal{H}_c(\Lambda)\subset \mathcal{H}(\Lambda)$ the set of functions satisfying the C-property. We have:
\begin{pro}
 Let $\Lambda$ be an axially symmetric open set.
\begin{enumerate}
\item
Let $f\in\mathcal{H}(\Lambda)$ and ${\bf i}, {\bf j}\in\mathbb{S}^2$ with ${\bf i}\perp {\bf j}$. Then there exist four functions $h_\ell\in  Hol_{c}(\Lambda_{\bf i})$, $\ell=0,\ldots, 3$ such that
$$
Q_{\bf i}[f]=h_0+h_1{\bf i}  +  h_2 {\bf j}  +h_3{\bf i} {\bf j}.
$$
\item Let $\{1,{\bf i},{\bf j},{\bf i}{\bf j}\}$ be a basis of $\mathbb{H}$. Then:
$$\mathcal{H}(\Lambda)=\mathcal{H}_c(\Lambda)  \oplus \mathcal{H}_c(\Lambda){\bf i}   \oplus  \mathcal{H}_c(\Lambda)  {\bf j}\oplus \mathcal{H}_c(\Lambda) \bf{ij}.
$$
\end{enumerate}
\end{pro}
\begin{proof}
Part 1 is immediate if one writes $f(x+{\bf i}y)=\alpha(x,y)+{\bf i}\beta(x,y)$ with $\alpha(x,y)=\alpha_0(x,y)+\alpha_1(x,y) {\bf i} +\alpha_2(x,y) {\bf j} +\alpha_3(x,y){\bf ij} $, $\beta(x,y)=\beta_0(x,y)+\beta_1(x,y) {\bf i} +\beta_2(x,y) {\bf j} +\beta_3(x,y){\bf ij}$. The Cauchy-Riemann equations implies that the functions $h_\ell=\alpha_\ell + {\bf i} \beta_\ell (x,y)$ are holomorphic (this fact has been discussed also \cite{css}); the conditions (\ref{e-o}) imply that $h_\ell\in Hol_{c}(\Lambda_{\bf i})$ for all $\ell=0,\ldots, 3$.\\
To prove the second statement, we use the previous step to write the restriction to $\mathbb{C}({\bf i})$ of  a function  $f\in\mathcal{H}(\Lambda)$ as
 $f(x+{\bf i}y)=h_0+h_1{\bf i}  +  h_2 {\bf j}  +h_3{\bf i} {\bf j}$. If we now set $\tilde h_\ell(x+I_qy):=\alpha_\ell(x,y)+I_q\beta_\ell (x,y)$ we have that $\tilde h_\ell\in\mathcal{H}_c(\Lambda)$. The fact that the sum is a direct sum can be obtained as in the proof of Proposition \ref{pro20}.
\end{proof}
\section{The slice regular Bergman space and the C-property}

We now apply the results of the previous sections to deduce some properties of the regular Bergman spaces.

\begin{defin} Let  $\Lambda\subset\mathbb H$ be  a bounded axially symmetric slice domain. We denote  by $\mathcal A^2(\Lambda_{\bf i})$   the holomorphic Bergman space associated to $\Lambda_{\bf i}$, and by $\mathcal A(\Lambda)$ the slice regular Bergman space  associated to $\Lambda$ i.e.
$$
\mathcal A(\Lambda)=\{f\in\mathcal{SR}(\Lambda)\ | \ \int_{\Lambda}|f|^2 d\mu < \infty\}.
$$
The set $\mathcal A(\Lambda)$ is equipped with the norm inherited from the $\mathcal L_2$-space.

By $\mathcal A_c(\Lambda) $ (resp. $\mathcal A_{\bar c}(\Lambda)$) we denote the real  {linear} subspace of $\mathcal A(\Lambda)$ whose elements  {satisfy} the quaternionic C-property   (resp.  the quaternionic anti-C-property).
\\
Finally, we consider the slice Bergman spaces
  $$
 \mathcal A(\Lambda_{{\bf i}}):=\{  f\in \mathcal{SR}(\Lambda) \   \mid  \
    \displaystyle \|f\|^2_{\mathcal A(\Lambda_{\bf i})}:= \int_{\Lambda_{\bf i}}|f_{\mid_{\Lambda_{{\bf i}}}}|^2 d\sigma_{\bf i} <\infty \}.
  $$
where $d \sigma_{\bf i}$ denotes the area element in the complex plane $\Lambda_{{\bf i}}$.
On $\mathcal A(\Lambda_{{\bf i}})$ we define the scalar product
$$
\langle f,g\rangle_{\mathcal A(\Lambda_{{\bf i}})}=\int_{\Lambda_{{\bf i}}} \overline{f}\,d\sigma_{{\bf i}}\, g.
$$
  \end{defin}
\begin{obs}{\rm
Proposition 3.2 in \cite{CGESS}  implies that  the map $P_{\bf i}$ restricted to $\mathcal A^2(\Lambda_{\bf i})$ gives the embedding $$P_{\bf i}\mid_{\mathcal A^2(\Lambda_{\bf i})}: \mathcal A^2(\Lambda_{\bf i}) \to \mathcal A(\Lambda).$$
}
\end{obs}

\begin{pro}
Let $\Lambda$ be an axially symmetric open set in $\mathbb{H}$. Given $f,g \in\mathcal A_c(\Lambda ) $, then
$$
\overline{\int_{\Lambda_{\bf i}} \overline{ f(\zeta)} d\sigma_{\zeta} g(\zeta) }=\int_{\Lambda_{\bf i}} \overline{ f(\zeta)} d\sigma_{\zeta} g(\zeta)\in\mathbb{R}.
$$
\end{pro}
\begin{proof}
It is a direct consequence of the C-property:  the values of $f,g$ commute by Corollary \ref{comm}  and the conjugation on the slice $\Lambda_{\bf i}$  preserves the differential  element of area. Thus, the inner product of two any elements of $\mathcal A_c(\Lambda ) $  has real values.
\end{proof}

\begin{pro}
Let $\Lambda\subseteq\mathbb{H}$ be an axially symmetric open set.
The space $\mathcal A_c(\Lambda ) $ is a real linear Hilbert space.
\end{pro}
\begin{proof}
It is similar to the complex case, see Proposition \ref{pro2}, and it is a direct consequence of the following facts, presented in \cite{CGESS}:
\begin{enumerate}
 \item  $\mathcal A(\Lambda ) $ is a complete normed space.
\item The convergence in the norm implies the uniform convergence on compact subsets of the slice $\Lambda_{\bf i}$.
\end{enumerate}
\end{proof}

Proposition \ref{pro20} implies the existence of  functions  $\mathcal K^n_{\Lambda}(\cdot, q)\in \mathcal A_c(\Lambda)$, for $n=0,1,2,3$, such that
$$
\mathcal K_{\Lambda}(\cdot, q) =\mathcal K_{\Lambda}^0(\cdot, q)  +\mathcal K_{\Lambda}^1(\cdot, q) {\bf i}+ \mathcal K_{\Lambda}^2(\cdot, q)  {\bf j} + \mathcal K_{\Lambda}^3(\cdot, q) {\bf ij},
$$
for any $q$ belonging to an axially symmetric s-domain $\Lambda$.
These functions are the components, which  {satisfy} the C-property, of the slice regular Bergman kernel.
\\

Let us now fix  the basis $\{e_0=1, e_1={\bf i} ,e_2={\bf j},e_3={\bf ij}\}$ of the algebra of quaternions.

\begin{pro}
Let $\Lambda$ be an axially symmetric s-domain and let $q\in\Lambda$.
Let any  $f = f_0 + f_1 e_1+ f_2e_2 +f_3e_3 \in \mathcal A_c(\Lambda) $, with $f_n$ real valued function for $n=0,1,2,3 $. Then
$$
f_n(q)=\int_{\Lambda_{\bf i}}  \mathcal K^n_{\Lambda}(q, \zeta) f(\zeta) d\sigma_\zeta,\quad \forall n=0,1,2,3.
$$
\end{pro}
\begin{proof}
Fist of all, let us recall that for any fixed $\mathbf{i}\in\mathbb{S}^2$ and $q\in \Lambda_{\bf i}$, the evaluation functional
$$
\phi_q[f]=f(q),\quad \forall f\in \mathcal A(\Lambda)
$$
is bounded  on $\mathcal A(\Lambda)$, see \cite{CGESS}.
Then, by setting  $\theta_n( q)  =\theta_n(  x_0+ x_1e_1 +x_2e_2 +x_3e_3)= x_n $, $n=0,1,2,3$,
one has that also the real linear functionals
$\theta_n\circ  \phi_q$     are   bounded  on    $\mathcal A_c(\Lambda)$ for  $n=0,1,2,3$.
By the Riesz representation theorem in the  {real linear} space   $\mathcal A_c(\Lambda)$, there exist $L_q^n\in\mathcal A_c(\Lambda)$ such that
$$
f_n(q)=\int_{\Lambda_{\bf i}}     \overline{L_q^n}(r) d\sigma_r f(r) ,\quad \textrm{for} \quad  n=0,1,2,3,
$$
where $d\sigma$ denotes the scalar element of area.
Then we have
$$
 e_n f_n(q)= \int_{\Lambda_{\bf i}}    e_n \overline{L_q^n }(r) d\sigma_r f(r),\quad \textrm{for} \quad  n=0,1,2,3.
 $$
Denoting $K^n(q,\cdot) =\overline{L_q^n}(\cdot)$  one has  that
$$
e_n f_n(q)= f_n(q) e_n= \int_{\Lambda_{\bf i}}     K^n(q,r)e_n d\sigma_r f(r),\quad \textrm{for} \quad  n=0,1,2,3,
$$
and setting
$$\displaystyle K(q,\cdot) =\sum_{n=0}^3 K^n(q,\cdot)e_n,$$
 we obtain
$$
f(q)= \int_{\Lambda_{\bf i}}   K(q,r)  d\sigma_r f(r).
$$
Note that given $r\in \Lambda_{\bf}$,  then  $\displaystyle \overline{K(r,\cdot)} =\sum_{n=0}^3 \overline{e_n} L_{r}^n  \in \mathcal A_{c}(\Lambda)$, and
$$
\overline{K(r, q)} = \int_{\Lambda_{\bf i}}   K(q,\zeta)  d\sigma_{\zeta}  \overline{K(r,\zeta)}.
$$
Applying the quaternionic conjugation one has that
$$
K(r, q) = \int_{\Lambda_{\bf i}}  K(r,\cdot)   d\sigma  \overline{ K(q,\cdot)} = \overline{ K(q,r)} .
$$
That fact proves that $ K(\cdot ,\cdot)  $ is hermitian.
Proposition \ref{pro20}, implies that any element in  $\mathcal A(\Lambda )$ is a quaternionic right linear combination of the elements of $\mathcal A_c(\Lambda )$, then $K(\cdot ,\cdot) $ is a reproducing kernel of  $\mathcal A(\Lambda )$. The uniqueness  of slice regular  Bergman  kernel implies the result.
\end{proof}

\subsection{On the reproducing property of the slice regular Bergman kernel}
Let us recall, see \cite{CGESS}, that the slice Bergman kernel of the first kind  $\mathcal B(\cdot, \cdot)$ is slice regular in its first coordinate and it is right anti-slice regular in its second coordinate. We can introduce the operator which takes the restriction of an anti-slice regular function to a complex plane $\mathbb{C}({\bf i})$. With an abuse of notation, we will denote this restriction operator by same symbol $Q_{\bf i}$ used for slice regular functions.
\begin{defin}
Let $\Lambda$ be an axially symmetric s-domain.
By $\mathcal B_{\Lambda}(\cdot,\cdot)$ we  denote the regular Bergman kernel of the first kind associated to $\Lambda$, and by $\mathcal K_{\Lambda}(\cdot, \cdot)$, denote the slice regular Bergman kernel of the second kind associated to $\Lambda$, see \cite{CGESS}. \\
Let us define the operator
$M_{\bf i}$ on $\mathcal A(\Lambda)$ as follows: let $f\in \mathcal A(\Lambda)$, then
\begin{equation}\label{equaM}
\displaystyle M_{\bf i}[f](q):=\int_{\Lambda_{\bf i}}  Q_{\bf i}[\mathcal B(q, \zeta )] Q_{\bf i}[f](\zeta)d\sigma_\zeta  , \quad \forall q \in \Lambda,
\end{equation}
where $ Q_{\bf i}[\mathcal B(q, \zeta )]$ acts as the restriction to $\mathbb{C}({\bf i})$ of the second coordinate.
\end{defin}

The next result says that, using a suitable operator denoted by $M_{\bf i}$,   { the regular Bergman kernel  of the second kind  reproduces all elements of the slice regular Bergman space not only on the slices,
as shown in \cite{CGESS},}  but on the whole domain $\Lambda$.

\begin{teo} Let $\Lambda$ be an axially symmetric s-domain. Then we have:
  \begin{equation}\label{equation6}
  \displaystyle f(q)=\int_{\Lambda} \mathcal K_{\Lambda}(q,r)   M_{\bf i}[f](r) d\mu_r, \quad \forall f \in \mathcal A(\Lambda),
      \end{equation}
  where $d\mu_r$ is the  differential element of volume.
\end{teo}
\begin{proof}
As we proved in \cite{CGESS}, for any $f\in \mathcal A(\Lambda)$ one has
\begin{equation}\label{equation8}
f(q)=\int_{\Lambda_{\bf i}}{\mathcal K_{\Lambda} (q, \zeta)} Q_{\bf i}[f](\zeta) d\sigma_\zeta=\int_{\Lambda_{\bf i}} \overline{\mathcal K_{\Lambda} (\zeta, q)} Q_{\bf i}[f](\zeta) d\sigma_\zeta,
\end{equation}
the properties of the Bergman kernel of the first kind one has
\begin{equation}
  \label{equation5}
  \mathcal K_{\Lambda}(\zeta, q) = \int_{\Lambda}  \mathcal B_{\Lambda}(\zeta ,r) \mathcal K_{\Lambda} (r, q) d\mu_{r},
\end{equation}
where $q,z\in \Lambda $.
By substituting  (\ref{equation5}) in (\ref{equation8}) we obtain
$$
f(q)=\int_{\Lambda_{\bf i}}  \int_{\Lambda}  \mathcal K_{\Lambda} (q,r ) \mathcal B_{\Lambda}(r, \zeta)  Q_{\bf i}[f](\zeta) d\sigma_\zeta d\mu_r
$$

$$
  =\int_{\Lambda} \mathcal K_{\Lambda}(q,r)  \int_{\Lambda_{\bf i}}  Q_{\bf i}[\mathcal{B}(r,\zeta)] Q_{\bf i}[f](\zeta) d\sigma_\zeta d\mu_r, \quad \forall f \in \mathcal A(\Lambda)
$$
 and recalling (\ref{equaM}) we get the result. \end{proof}

 {The operator $M_{\bf i}$ has the following properties with respect to the inner product of the slice regular Bergman space:}
\begin{pro}
 Let $\Lambda$ be an axially symmetric s-domain. Then:
  \begin{enumerate}
    \item $\displaystyle \int_{\Lambda} \overline{M_{\bf i}[f]}gd\mu =\int_{\Lambda_{\bf i}}\overline{  Q_{\bf i}[f]} Q_{\bf i} [g] d\sigma, \quad  \forall f, g\in \mathcal A(\Lambda)$,
  \item $\displaystyle\int_{\Lambda} \overline{M_{\bf i}[g] }g d\mu  =\int_{\Lambda_{\bf i}} |Q_{\bf i}[g] |^2 d\sigma ,\quad \forall g\in \mathcal A(\Lambda)$,
     \end{enumerate}
where $d\mu$ is the  differential element of volume and $d\sigma$  is the differential element of area.
\end{pro}
\begin{proof} It follows from the relation
\begin{enumerate}
      \item
$$\begin{array}{l}\displaystyle \int_{\Lambda} \overline{M_{\bf i}[f](\xi)}g(\xi)d\mu_{\xi} =
\int_{\Lambda}    \overline{\left(  \int_{\Lambda_{\bf i}}  Q_{\bf i}  [\mathcal B(\xi, \zeta )] Q_{\bf i}[f](\zeta) d\sigma_\zeta \right) } g(\xi)d\mu_\xi\\
\\
\displaystyle
=\int_{\Lambda_{\bf i}} \overline{ Q_{\bf i}[f](\zeta)  }  \left(  \int_{\Lambda}      \mathcal B(\zeta, \xi )  g(\xi)d\mu_\xi  \right)   d\sigma_\zeta  =
 \int_{\Lambda_{\bf i}} \overline{ Q_{\bf i}[f](\zeta)  } Q_{\bf i}[ g](\zeta)     d\sigma_\zeta.
\end{array}$$

\item By setting $f=g$ in the previous identity we get the result.
     \end{enumerate}
\end{proof}


\begin{thebibliography}{10}

\bibitem{berg} S. Bergman,  \emph{The Kernel Function and Conformal Mapping}, American Mathematical Society, Providence, RI, 1970.

\bibitem{bgsc} S. Bergman, M. Schiffer,  \emph{Kernel Functions and Elliptic Differential Equations in Mathematical Physics}, Accademic Press, New York, 1953.

\bibitem{bradel} Brackx F., R. Delanghe, \emph{Hypercomplex function theory and Hilbert modules with reprocucing kernel},
Proc. Amer. Math. Soc., {\bf 37} (1978), 545--576.

\bibitem{bds} F. Brackx, R. Delanghe, F. Sommen, {\em Clifford Analysis},
Pitman Res. Notes in Math., 76, 1982.

\bibitem{CGESS} {} F. Colombo, J. O. Gonz\'alez-Cervantes, M. E. Luna-Elizarrar\'as, I. Sabadini, M. Shapiro, \emph{On two approaches to the Bergman theory for slice regular functions},
Advances in Hypercomplex Analysis; Springer INdAM Series 1, (2013), 39--54.

\bibitem{global}  F. Colombo, J. O. Gonz\'alez-Cervantes, I. Sabadini, \emph{A non constant coefficient differential operator associated to slice monogenic functions},
 Transactions of the American Mathematical Society, {\bf 365} (2013), 303-318.


\bibitem{csTrends} F. Colombo,  I. Sabadini, \emph{A structure formula for  slice monogenic functions and some of its
consequences}, Hypercomplex Analysis, Trends in Mathematics, Birkh\"auser, 2009, 101--114.

\bibitem{slicecss} F. Colombo, I. Sabadini, D.C. Struppa, \emph{
 Slice monogenic functions}, Israel J. Math.  {\bf 171} (2009), 385--403.


\bibitem{cssbook} F. Colombo, I. Sabadini, D.C. Struppa, \emph{Noncommutative Functional Calculus, Theory and Applications of Slice Hyperholomorphic Functions},
Progress in Mathematics V. 289, Birkh\"auser Basel 2011.

\bibitem{css} F. Colombo, I. Sabadini, D.C. Struppa,
 \emph{Sheaves of  slice regular functions}, Math. Nach., {\bf 285} (2012), 949-958



\bibitem{const} D. Constales, \emph{ The Bergman and Szeg\"o kernels for separately monogenic functions},
Zeit. Anal. Anwen.,  {\bf 9} (1990), 97--103.

\bibitem{conskrau} D. Constales, R. S. Krau\ss har,  \emph{Bergman kernels for rectangular domains and multiperiodic functions in Clifford analysis},  Math. Meth. Appl. Sci., {\bf 25} (2002), 1509--1526.

  \bibitem{conskrau2} D. Constales, R. S. Krau\ss har,  \emph{Bergman spaces of higher-dimensiononal hyperbolic polyhedron-type domains I},  Math. Meth. Appl. Sci., {\bf 29} (2006), 85--98.

\bibitem{delanghe} R. Delanghe, \emph{On Hilbert modules with reproducing kernel}, Functional and Theoretical Methods:Partial Differential Equations, Proceedings of the International Symposium. DaRMSTADT, 1976, Lecture Notes in Mathematics, {\bf 561} (1976), 158--170.



\bibitem{advances} G. Gentili, D.C. Struppa, \emph{ A new theory of regular functions
of a quaternionic variable},  Adv. Math., {\bf 216} (2007),
279--301.

\bibitem{gp} R. Ghiloni, A. Perotti, \emph{ A new approach to slice regularity on real algebras},
Adv. Math., {\bf 226} (2011), 1662--1691.



\bibitem{SG2} {} J. O. Gonz\'alez-Cervantes, M. E. Luna-Elizarrar\'as, M. Shapiro, \emph{On Some Categories and Functors in the Theory of Quaternionic Bergman Spaces},  Advances in Applied Clifford Algebras,   {{\bf 19} (2009) 325--338.}

\bibitem{K}    {}   S. G. Krantz, \emph{Function Theory  of several complex variables}, Wadsworth $\&$ Brooks, 1982, 2nd edition.


\bibitem{shavas}  M. Shapiro, N. Vasilevski, \emph{ On the Bergman kernel function in hyperholomorphic analysis},
Acta Appl. Math., {\bf 46} 1977, 1--27.

 \bibitem{shavas2} M. Shapiro, N. Vasilevski, \emph{On the Bergman kernel function in Clifford analysis}, In
 Clifford Analysis and Their Applications in Mathematical Physics, Brackx F. et al (eds.). Proceedings of the Third Conference, Deinze,
 Belgium, 1993. Dordrecht: Kluver Academic Publisher. Fundamental Theories of Physics, {\bf 55} (1993), 183--192.

 \bibitem{shavas3} M. Shapiro, N. Vasilevski, \emph{On the Bergman kernel functions in quaternionic analysis},
 Russian Mathematica, {\bf 42} (1998), 81--85.













\end{thebibliography}
\end{document}